\documentclass[12pt]{amsart}
\usepackage{SSdefn}

\newcommand{\invl}{\llbracket\hspace{-.12cm} \llbracket}
\newcommand{\invr}{\rrbracket\hspace{-.12cm} \rrbracket}
\newcommand{\lpp}{(\!(}
\newcommand{\rpp}{)\!)}
\newcommand{\vars}{\invl x_1,x_2,\dots\invr}

\title{Big polynomial rings with imperfect coefficient fields}

\author{Daniel Erman}
\address{Department of Mathematics, University of Wisconsin, Madison, WI}
\email{\href{mailto:derman@math.wisc.edu}{derman@math.wisc.edu}}
\urladdr{\url{http://math.wisc.edu/~derman/}}

\author{Steven V Sam}
\address{Department of Mathematics, University of Wisconsin, Madison, WI}
\curraddr{Department of Mathematics, University of California, San Diego, La Jolla, CA}
\email{\href{mailto:ssam@ucsd.edu}{ssam@ucsd.edu}}
\urladdr{\url{http://math.ucsd.edu/~ssam/}}

\author{Andrew Snowden}
\address{Department of Mathematics, University of Michigan, Ann Arbor, MI}
\email{\href{mailto:asnowden@umich.edu}{asnowden@umich.edu}}
\urladdr{\url{http://www-personal.umich.edu/~asnowden/}}

\thanks{DE was partially supported by NSF DMS-1601619. SS was partially supported by NSF DMS-1500069. AS was partially supported by NSF DMS-1453893.}

\date{May 6, 2020}

\begin{document}

\maketitle

\begin{abstract}
We previously showed that the inverse limit of standard-graded polynomial rings with perfect (or semi-perfect) coefficient field is a polynomial ring, in an uncountable number of variables. In this paper, we show that the result holds with no hypothesis on the coefficient field.  We also prove an analogous result for ultraproducts of polynomial rings.
\end{abstract}

\section{Introduction}

\subsection{Statement of results}

Let $\bk$ be a field and let $\bR$ be the inverse limit of the standard-graded polynomial rings $\bk[x_1, \ldots, x_n]$ in the category of graded rings; thus $\bR$ is a graded ring, and a degree $d$ element of $\bR$ is a formal, perhaps infinite, $\bk$-linear combination of degree $d$ monomials in the variables $\{x_i\}_{i \ge 1}$. Recall that $\bk$ is {\bf perfect} if it has characteristic~0, or if it has positive characteristic~$p$ and $\bk^p=\bk$. We say that $\bk$ is {\bf semi-perfect} if it has characteristic~0, or if it has characteristic~$p$ and $[\bk:\bk^p]<\infty$. In \cite{stillman}, we proved that $\bR$ is (isomorphic to) a polynomial ring (in uncountably many variables) when the field $\bk$ is semi-perfect, and we demonstrated the utility of this result by using it to give a new proof of Stillman's conjecture; it has also been used by \cite{draisma-lason-leykin} to prove a finiteness theorem for Gr\"obner bases. In this paper, we improve our result by eliminating the hypothesis:

\begin{theorem} \label{mainthm}
For any field $\bk$, the ring $\bR$ is a polynomial ring.
\end{theorem}

As in \cite{stillman}, we prove an analogous result for an ultraproduct ring, and establish some related results; see \S \ref{s:app} for details.

\subsection{Motivation}

We offer three pieces of motivation for Theorem~\ref{mainthm}:

(a) Ananyan--Hochster \cite[Theorem B]{ananyan-hochster} proved the existence of ``small subalgebras'' over algebraically closed fields. Utilizing these algebras was one of the key insights in their proof of Stillman's conjecture.  In~\cite{stillman}, we extended the theorem of Ananyan--Hochster by establishing the existence of small subalgebras over perfect fields. Theorem ~\ref{mainthm} allows us to further strengthen this result to all fields. See \S \ref{ss:small} for details.

(b) Ananyan--Hochster introduced a notion of ``strength'' for polynomials that played a central role in their work, and has since featured prominently elsewhere (e.g., \cite{ananyan-hochster, bik-draisma-eggermont, draisma,stillman,hartshorne}). Since strength has proven to be such a useful concept, it is desirable to understand it better. There are some genuine differences as to how strength behaves when the coefficient field is not semi-perfect. For example, results of \cite{ananyan-hochster} imply that if $\bk$ is algebraically closed then an element $f \in \bR$ has infinite strength if and only if the ideal of partial derivatives $(\frac{\partial f}{\partial x_1}, \frac{\partial f}{\partial x_2},\dots)$ has infinite codimension. This is not true if $\bk$ fails to be semi-perfect: indeed, if $a_1,a_2,\dots\in \bk$ are linearly independent over $\bk^p$ then $f=\sum_{i=1}^\infty a_ix_i^p$ is an infinite strength element whose corresponding ideal of partial derivatives is the zero ideal. Theorem~\ref{mainthm} shows nonetheless that some of the pleasant features of strength in the semi-perfect case continue to hold for general fields: indeed, Theorem~\ref{mainthm} is equivalent to the statement that if $f_1, \ldots, f_r$ are homogeneous elements of $\bR_+$ such that no homogeneous linear combination has finite strength then $f_1, \ldots, f_r$ are algebraically independent.

(c) In \cite[\S 5]{stillman}, we gave a geometric proof of Stillman's conjecture. The basic idea is as follows. Let $X$ be the space of tuples $(f_1, \ldots, f_r)$ where $f_1, \ldots, f_r$ are homogeneous polynomials in infinitely many variables of fixed degrees $d_1, \ldots, d_r$. At each point $x=(f_1, \ldots, f_r)$ in $X$ there is a corresponding ideal $I_x=(f_1, \ldots, f_r)$ in the polynomial ring. Using the polynomiality result for $\bR$ (or, really, the bounded version $\bR^{\flat}$ discussed in \S \ref{ss:bounded}), we proved a generic version of Stillman's conjecture: for any closed subset $Z$ of $X$, the ideal $I_x$ has bounded projective dimension for generic $x \in Z$. Appealing to Draisma's theorem \cite{draisma} that $X$ is $\GL_{\infty}$-noetherian, we then deduced Stillman's conjecture from the generic version.

The generic version of Stillman's conjecture for $Z$ involves the ring $\bR^{\flat}$ where $\bk$ is the function field of $Z$. In positive characteristic, this field is typically not semi-perfect: for example, when $Z$ is the entire space $X$, the field $\bk$  is a rational function field in infinitely many indeterminates. Since our previous polynomiality result did not hold in this setting, our proof had some extra steps: we passed to the algebraic closure of $\bk$ (which is perfect), carried out our argument using the polynomiality of $\bR^{\flat}$ there, and then descended. Theorem~\ref{mainthm} is partially motivated by the desire to eliminate this complication. We are only partially successful: we show that $\bR^{\flat}$ is polynomial in some situations (such as when $Z=X$), but not all the ones used in our geometric proof of Stillman's conjecture.

Nevertheless, as we believe that the general strategy we used in our geometric approach to Stillman's conjecture could be useful in other situations, it is worthwhile to try to simplify the details. Theorem~\ref{mainthm} (and its generalizations) are a significant step in this direction.

\subsection{Overview of proof}

Our proof of Theorem~\ref{mainthm} is an adaptation of the method used in \cite[\S2]{stillman} to treat the perfect and semi-perfect cases, so we first recall that. The main idea is to characterize polynomial rings using derivations in a manner that can be applied to $\bR$. To describe this abstract characterization, let $R$ be a graded ring with $R_0=\bk$ a field.

First suppose that $\bk$ has characteristic~0. We say that $R$ has {\bf enough derivations} if for every non-zero homogeneous element $x$ of positive degree, there is a homogeneous derivation $\partial$ of $R$ of negative degree with $\partial(x) \ne 0$. We prove \cite[Theorem~2.2]{stillman} that $R$ is a polynomial ring if and only if it has enough derivations. It is easy to see that the inverse limit ring $\bR$ has enough derivations: in fact, derivatives with respect to the variables are all one needs. Thus $\bR$ is a polynomial ring.

Now suppose that $\bk$ has positive characteristic $p$. Since any derivation annihilates any $p$th power, we cannot carry over our previous characterization of polynomial rings verbatim. Recall that a {\bf Hasse derivation} of $R$ is sequence $\{\partial^n\}_{n \ge 0}$ of linear endomorphisms of $R$ such that $\partial^1$ is a derivation and $\partial^n$ behaves like $\tfrac{1}{n!} (\partial^1)^n$; see Definition~\ref{def:hasse} for the exact definition. We say that $R$ has {\bf enough Hasse derivations} if for every homogeneous element $x$ that is not in the $\bk$-span of the set $R^p$, there is a homogeneous Hasse derivation $\partial$ of negative degree such that $\partial^1(x) \ne 0$. We show \cite[Theorem~2.11]{stillman} that $R$ is a polynomial ring if and only if it has enough Hasse derivations. This abstract theorem does not require $\bk$ to be perfect; see \cite[Remark~2.12]{stillman}. When $\bk$ is perfect, it is easy to see that $\bR$ has enough Hasse derivations: the Hasse derivatives with respect to the variables are all one needs. Thus, in this case, we see that $\bR$ is a polynomial ring. (This reasoning can be extended to the case where $\bk$ is semi-perfect, see \cite[Remark~5.4]{stillman}.)

Here is the main kind of problem that arises when $\bk$ is not semi-perfect. Suppose that $\bk=\bF_p(t_1, t_2, \ldots)$ is the field of rational functions in the infinitely many variables $\{t_i\}_{i \ge 1}$, which is not semi-perfect. Consider the element $f=\sum_{i \ge 1} t_i x_i^p$ of $\bR$. This element $f$ is not a $p$th power, or even a $\bk$-linear combination of $p$th powers. Thus, if we wanted to prove $\bR$ is a polynomial ring using the criterion of the previous paragraph, we would need to produce a Hasse derivation $\partial$ such that $\partial^1(f) \ne 0$. However, if $\partial^1$ is continuous with respect to the inverse limit topology then it commutes with the infinite sum defining $f$, and so $\partial^1(f)=0$. Since all of the obvious Hasse derivations of $\bR$ are continuous, it is not clear how to proceed.

Our strategy is to give a new characterization of polynomial rings via Hasse derivations that can accommodate the issue seen above. The ``problem elements'' in $\bR$ (i.e., those annihilated by $\partial^1$ for all the obvious Hasse derivations $\partial$) are exactly those elements in which all variables appear with exponent divisible by $p$. When $\bk$ is perfect, this is exactly $\bR^p$; when $\bk$ is semi-perfect, it is exactly the $\bk$-span of $\bR^p$. In general, however, these elements cannot be directly characterized from the $\bk$-algebra structure on $\bR$. We therefore introduce a new piece of structure that detects these elements. Define ring endomorphisms $\phi$ and $\sigma$ of $\bR$ by
\begin{displaymath}
\phi \left( \sum c_e x^e \right) = \sum c_e^p x^e, \qquad
\sigma \left( \sum c_e x^e \right) = \sum c_e x^{pe},
\end{displaymath}
where the sum is over multi-indices $e$. Thus $\phi$ raises the coefficients to the $p$th power, while $\sigma$ raises the variables to the $p$th power. The ``problem elements'' are exactly those elements in the image of $\sigma$. For example, the element $f$ defined above is $\sigma(\sum_{i \ge 1} t_i x_i)$.

Returning to the abstract setting, we define an {\bf $F$-factorization} on the graded ring $R$ to be a pair $(\phi, \sigma)$ similar to the above: $\phi$ is degree preserving, $\sigma$ is $\bk$-linear, and $\phi \circ \sigma = \sigma \circ \phi$ is the Frobenius map $F$. (There is one additional condition we demand of an $F$-factorization; see Definition~\ref{def:ffac}.) Fix such a structure on $R$, and let $\cD$ be some collection of Hasse derivations on $R$.  We define a notion of admissibility for $\cD$ (Definition~\ref{def:admiss}). One of the key conditions is (or implies) that for every homogeneous element of $R$ that is not in the image of $\sigma$, there exists a Hasse derivation $\partial\in \cD$ of negative degree such that $\partial^1(x)\ne 0$.

We then show (Theorem~\ref{thm:poly}) that $R$ is a polynomial ring if and only if it admits an $F$-factorization with an admissible set $\cD$ of Hasse derivations. The proof of Theorem~\ref{thm:poly}, which is the heart of the paper, is similar to the proof of \cite[Theorem~2.11]{stillman}. Essentially, we consider a hypothetical algebraic relation among a minimal generating set, and use one of the given Hasse derivations to produce a relation of lower degree, eventually yielding a contradiction.  The argument in this paper is somewhat more complicated due to the more limited properties of the Hasse derivations in the set $\cD$.

Returning to $\bR$, we have already defined an $F$-factorization on it. Let $\cD$ be the set of Hasse derivatives with respect to the variables $x_i$. It will not be too difficult to verify that $\cD$ is admissible, and so we conclude that $\bR$ is a polynomial ring.  In essence, Theorem~\ref{thm:poly} enables us to work with a smaller set of Hasse derivatives than \cite[Theorem~2.11]{stillman}, thereby bypassing any need to construct discontinuous Hasse derivations on $\bR$.

\subsection{Notation and conventions}

Throughout, ``graded'' means graded by the non-negative integers. The symbol $p$ will always denote a prime number. Most of the rings we consider will have characteristic~$p$.

\section{$F$-factorizations}

For a ring $R$ of characteristic $p$, let $F=F_R$ be the $p$th power homomorphism $R \to R$. When $R=\bk[x_i]_{i \in I}$ is a polynomial ring, $F$ can be factored into two separate operations: raising the coefficients to the $p$th power, and raising the monomials to the $p$th power. The following definition is an abstraction of this.

\begin{definition} \label{def:ffac}
Let $R$ be a graded ring of characteristic~$p$. An {\bf $F$-factorization} on $R$ is a pair $(\phi, \sigma)$ satisfying the following:
\begin{itemize}
\item[(F1)] $\phi \colon R \to R$ is an injective homomorphism of graded rings satisfying $\phi \vert_{R_0} = F_{R_0}$.
\item[(F2)] $\sigma \colon R \to R$ is a ring homomorphism satisfying $\sigma \vert_{R_0} = \id_{R_0}$.
\item[(F3)] The maps $\phi$ and $\sigma$ commute and satisfy $\phi \circ \sigma = F_R$.
\item[(F4)] For any $\epsilon_1, \ldots, \epsilon_s \in R_0$ we have
\begin{displaymath}
\im(\sigma) \cap \sum_{j=1}^s \epsilon_j \im(\phi) = \sum_{j=1}^s \epsilon_j R^p.  \qedhere
\end{displaymath}
\end{itemize}
\end{definition}

\begin{remark}
Let $(\phi, \sigma)$ be an $F$-factorization on $R$. Then we have the following:
\begin{enumerate}
\item If $a \in R_0$ and $x \in R$, then $\phi(ax)=a^p \phi(x)$.
\item If $a \in R_0$ and $x \in R$, then $\sigma(ax)=a \sigma(x)$, that is, $\sigma$ is a homomorphism of $R_0$-algebras.
\item If $x \in R_0$ is homogeneous of degree $d$, then $\sigma(x)$ is homogeneous of degree $pd$. Indeed, $\phi$ is injective and preserves degree, and $\phi(\sigma(x))=x^p$ is homogeneous of degree $pd$. \qedhere
\end{enumerate}
\end{remark}

\begin{remark} \label{rmk:F4}
Regarding condition (F4), we note that the inclusion $\supseteq$ follows from the other axioms.  The key content of the axiom is thus the other direction, which relates elements in the $\bk$-span of $R^p$ to $\sigma$ and $\phi$.
\end{remark}

\begin{proposition} \label{prop:standard-fact}
Let $S$ be a graded ring of characteristic~$p$ with $S_0=\bF_p$, let $\bk$ be a field of characteristic~$p$, and let $R=\bk \otimes_{\bF_p} S$. Put $\phi=F_{\bk} \otimes \id_S$ and $\sigma=\id_{\bk} \otimes F_S$. Then $(\phi, \sigma)$ is an $F$-factorization on $R$.
\end{proposition}

\begin{proof}
Conditions (F1), (F2), and (F3) are easy to see directly. We prove (F4). Let $\epsilon_1, \dots, \epsilon_s \in R_0$; we may assume, without loss of generality, that they are linearly independent over $R_0^p$. The inclusion $\sum_j \epsilon_j R^p \subseteq \im(\sigma) \cap \sum_{j=1}^s \epsilon_j \im(\phi)$ holds since every element in $R^p$ is in the image of both $\sigma$ and $\phi$. Conversely, suppose that $f$ is an element of $\im(\sigma) \cap \sum_j \epsilon_j \im(\phi)$. Write $f=\sum a_e x^e$, where the sum is over multi-indices $e$ and $a_e \in \bk$. Since $f \in \im(\sigma)$, we have $a_e=0$ unless $p \mid e$. Since $f \in \sum_j \epsilon_j \im(\phi)$, we can write $f= \sum_j \epsilon_j \phi(g_j)$ for some $g_j \in R$. Write $g_j = \sum b_{j,e} x^e$. Equating coefficients, we find $a_e = \sum_{j=1}^s \epsilon_j b_{j,e}^p$. If $p \nmid e$ then this vanishes; since the $\epsilon_j$ are linearly independent over $\bk^p$, it follows that each $b_{j,e}$ vanishes as well. We conclude that all monomials appearing in $g_j$ with non-zero coefficient are $p$th powers; in other words, we can write $g_j=\sigma(h_j)$ where $h_j=\sum b_{j,pe} x^e$. We thus find $f=\sum_{j=1}^s \epsilon_j \phi(\sigma(h_j)) \in \sum_{j=1}^s \epsilon_j R^p$.
\end{proof}

\begin{example} \label{ex:pfac}
Let $R=\bk[x_i]_{i \in I}$ be a polynomial ring, where each $x_i$ is homogeneous of positive degree. Letting $S=\bF_p[x_i]_{i \in I}$, we have $R=\bk \otimes_{\bF_p} S$. Proposition~\ref{prop:standard-fact} provides an $F$-factorization on $R$ which we call the {\bf standard $F$-factorization} on $R$. We note that it depends on the choice of variables, i.e., it is not invariant under automorphisms of $R$ (in the category of graded $\bk$-algebras).
\end{example}

\begin{proposition} \label{prop:pfac-reduced}
Let $R$ be a graded ring with an $F$-factorization $(\phi, \sigma)$. Then $R$ is reduced if and only if $\sigma$ is injective.
\end{proposition}

\begin{proof}
The ring $R$ is reduced if and only if $F=F_R$ is injective. Since $F$ factors as $\phi \circ \sigma$ and $\phi$ is injective, we see that $F$ is injective if and only if $\sigma$ is injective.
\end{proof}

\begin{proposition} \label{prop:pfac2}
Let $R$ be a reduced graded ring with an $F$-factorization $(\phi, \sigma)$. Let $\epsilon_1, \ldots, \epsilon_s \in R_0$ and let $x \in R$. Suppose that $\sigma(x) \in \sum_{j=1}^s \epsilon_j \im(\phi)$. Then $x \in \sum_{j=1}^s \epsilon_j \im(\phi)$.
\end{proposition}

\begin{proof}
The element $\sigma(x)$ belongs to $\im(\sigma) \cap \sum_{j=1}^s \epsilon_j \im(\phi)$, which is equal to $\sum_{j=1}^s \epsilon_j R^p$ by (F4). We can thus write $\sigma(x)= \sum_{j=1}^s \epsilon_j y_j^p$ for elements $y_j \in R$. Since $y_j^p=\sigma(\phi(y_j))$ and $\sigma$ is injective (Proposition~\ref{prop:pfac-reduced}), we see that $x = \sum_{j=1}^s \epsilon_j \phi(y_j)$.
\end{proof}

\begin{definition} \label{defn:level}
Let $R$ be a graded ring with an $F$-factorization $(\phi, \sigma)$. We define a decreasing filtration $\{ \cF^r R \}_{r \ge 0}$ on $R$ by $\cF^r R=\im(\sigma^r)$. We call this the {\bf level filtration}. We define the {\bf level} of $x \in R$ to be the supremum of the set $\{ r \in \bN \mid x \in \cF^r R\}$. For a subspace $V$ of $R$, we let $\cF^r V = V \cap \cF^r R$.
\end{definition}

\begin{remark}
  If $x$ is non-zero and of positive degree $d$, and $p^t$ is the maximum power of $p$ dividing $d$, then $x$ has level at most $t$. Thus the level filtration on $R_d$ is finite, that is, $\cF^r R_d=0$ for $r \gg d$ (in fact, $r>t$).
\end{remark}

\section{Hasse derivations}

\begin{definition} \label{def:hasse}
Let $R$ be a ring. A {\bf Hasse derivation} on $R$ is a collection $\partial=\{\partial^n\}_{n \in \bN}$ of additive endomorphisms of $R$ such that $\partial^0=\id$ and the identity
\begin{displaymath}
\partial^n(xy) = \sum_{i+j=n} \partial^i(x) \partial^j(y)
\end{displaymath}
holds for all $x,y \in R$. 
If $R$ has characteristic~$p$, we write $\partial^{[r]}$ in place of $\partial^{p^r}$.
\end{definition}

\begin{definition}
Let $R$ be a graded ring and let $\partial$ be a Hasse derivation on $R$. We say that $\partial$ is {\bf homogeneous of degree $d$} if for all homogeneous elements $x$ of $R$ the element $\partial^n(x)$ is homogeneous of degree $\deg(x)+nd$.
\end{definition}

\begin{example}
Let $R=\bk[x]$. Define $\partial^n$ to be the endomorphism of $R$ mapping $x^k$ to $\binom{k}{n} x^{k-n}$; here, as always, $\binom{k}{n}=0$ for $k<n$. Then $\partial$ is a Hasse derivation. Indeed, it suffices to check the defining identity on monomials. We have
\begin{displaymath}
\partial^n(x^a x^b)=\binom{a+b}{n} x^{a+b-n}, \qquad
\sum_{i+j=n} \partial^i(x^a) \partial^j(x^b) = \sum_{i+j=n} \binom{a}{i} \binom{b}{j} x^{a+b-n},
\end{displaymath}
which are equal by a standard identity on binomial coefficients (think about choosing a subset of size $n$ from $\{1,\dots,a\} \amalg \{a+1,\dots,a+b\}$).  We call $\partial$ the {\bf Hasse derivative}.

More generally, suppose $R=\bk[x_i]_{i \in I}$ is a multivariate polynomial ring, and pick $j \in I$. Then the Hasse derivative on $\bk[x_j]$ can be extended to $R$ by making it linear over $\bk[x_i]_{i \in I \setminus \{j\}}$. We call this Hasse derivation on $R$ the Hasse derivative with respect to $x_j$. If $R$ is graded and $x_j$ is homogeneous of degree $d$ then the Hasse derivative with respect to $x_j$ is homogeneous of degree $-d$.
\end{example}

\begin{proposition} \label{prop:hassepower}
Let $\partial$ be a Hasse derivation on the ring $R$, let $x \in R$, and let $k>0$. Then $\partial^k(x^n)$ can be expressed as a polynomial in $x$ of degree $\le n-1$ with coefficients in the subalgebra $\bZ[\partial^i(x)]_{1 \le i \le k}$.
\end{proposition}

\begin{proof}
If $n=1$ there is nothing to prove. Otherwise, we write
\begin{displaymath}
\partial^k(x^n) = \sum_{i+j=k} \partial^i(x^{n-1}) \partial^j(x),
\end{displaymath}
and the result follows by induction on $n$ and $k$.
\end{proof}

\begin{proposition} \label{prop:hasse1}
Let $R$ be a ring of characteristic~$p$. Let $\partial$ be a Hasse derivation on $R$ and let $x \in R$. Then
\begin{displaymath}
\partial^{[r]}(F^s{x}) = \begin{cases}
F^s(\partial^{[r-s]}{x}) & \text{if $r \ge s$} \\
0 & \text{if $r \le s$} \end{cases}.
\end{displaymath}
\end{proposition}

\begin{proof}
See \cite[Lemma 2.9]{stillman}.
\end{proof}

The following corollary of the above proposition will be used constantly without reference:

\begin{corollary} \label{cor:partial-preserve}
Let $R$ have characteristic $p$, and let $\epsilon_1, \ldots, \epsilon_s \in R_0$. Then for any $R_0$-linear Hasse derivation $\partial$, the space $\sum_{j=1}^s \epsilon_j R^p$ is killed by $\partial^{[0]}$ and stable under $\partial^{[r]}$ for all $r \ge 0$.
\end{corollary}

\section{Admissible sets}

Fix a graded ring $R$ with $R_0$ a field of characteristic~$p$ and an $F$-factorization $(\phi, \sigma)$ on $R$.

\begin{definition} \label{def:admiss}
Let $\cD$ be a set of Hasse derivations on $R$. We say that $\cD$ is {\bf admissible} (with respect to $(\phi,\sigma)$) if the following conditions are satisfied:
\begin{itemize}
\item[(D1)] Each $\partial \in \cD$ is homogeneous of negative degree.
\item[(D2)] Each $\partial \in \cD$ commutes with $\phi$.
\item[(D3)] Let $x \in R$ be homogeneous, and let $\epsilon_1, \ldots, \epsilon_s \in R_0$ (we allow $s=0$). Suppose that $\partial^{[0]}(x) \in \sum_{j=1}^s \epsilon_j \im(\phi)$ for all $\partial \in \cD$. Then $x \in \im(\sigma) + \sum_{j=1}^s \epsilon_j \im(\phi)$. \qedhere
\end{itemize}
\end{definition}

\begin{remark}
Suppose that $\bk$ is a perfect field of characteristic~$p$, that $R$ has the form $\bk \otimes_{\bF_p} S$ for some graded ring $S$ with $S_0=\bF_p$, and that $R$ is equipped with the $F$-factorization constructed in Proposition~\ref{prop:standard-fact}. Then $\phi$ is surjective and $\im(\sigma)=R^p$. Thus condition (D3) is vacuous if some $\epsilon_i$ is non-zero.  If all $\epsilon_i$ vanish (or if $s=0$) then it simply states that $\bigcap_{\partial \in \cD} \ker(\partial^1) = R^p$. Furthermore, a Hasse derivation of $R$ commutes with $\phi$ if and only if it comes from a Hasse derivation of $S$ by extension of scalars. Thus, $R$ admits an admissible set of Hasse derivations if and only if $S$ has enough Hasse derivations in the sense of \cite[Definition 2.10]{stillman}. (In fact, using \cite[Theorem~2.11]{stillman}, one can show that $S$ has enough Hasse derivations if and only if $R$ does.)
\end{remark}

\begin{example} \label{ex:hasse}
Suppose $R=\bk[x_i]_{i \in I}$ and $(\phi, \sigma)$ is the standard $F$-factorization. Let $\cD$ be the collection of Hasse derivatives with respect to the variables. We claim that this is admissible. Conditions (D1) and (D2) are immediate, so we check (D3). Pick $\epsilon_1,\dots,\epsilon_s \in R_0$ and suppose that $\partial^{[0]}(f) \in \sum_{j=1}^s \epsilon_j \im(\phi)$ for all Hasse derivatives $\partial$. Write $f=\sum c_e x^e$, where the sum is over multi-indices $e$, and let $V \subset \bk$ be the $\bk^p$-span of $\epsilon_1, \ldots, \epsilon_s$. Then $\sum_{j=1}^s \epsilon_j \im(\phi)$ consists of those elements of $R$ for which all coefficients belong to $V$, while $\im(\sigma)$ consists of those elements of $R$ for which all exponents are divisible by $p$. Thus to show that $f$ belongs to $\im(\sigma)+\sum_{j=1}^s \epsilon_j \im(\phi)$, we must show that if $e$ is not divisible by $p$ then $c_e$ belongs to $V$. Let such an $e$ be given, and pick $i \in I$ such that $e_i$ is not divisible by $p$. Let $\partial$ be the Hasse derivative with respect to $x_i$. Then $\partial^{[0]}(x^e)=e_i x^{e'}$, where $e'$ is obtained from $e$ by decrementing the $i$th entry by one. By assumption, $\partial^{[0]}(f)$ belongs to $\sum_{j=1}^s \epsilon_j \im(\phi)$, and so all coefficients of $\partial^{[0]}(f)$ belong to $V$. The coefficient of $x^{e'}$ is exactly $e_i c_e$. Since $e_i$ is a non-zero element of $\bk^p$, as it is an integer not divisible by $p$, it follows that $c_e \in V$, and so the result follows.
\end{example}

For the remainder of this section, we fix an admissible set $\cD$, and deduce some simple consequences of the definitions.

\begin{proposition}
Each $\partial \in \cD$ is $R_0$-linear.
\end{proposition}

\begin{proof}
Let $a \in R_0$ and $x \in R$. We have $\partial^n(ax)=\sum_{i+j=n} \partial^i(a) \partial^j(x)$. We have $\deg(\partial^i(a))=i \deg(\partial)$, which is negative for $i>0$ and thus vanishes. Hence, $\partial^n(ax)=a \partial^n(x)$.
\end{proof}

\begin{proposition} \label{prop:hassesigma}
Let $\partial \in \cD$. Then
\begin{displaymath}
\partial^{[s]}(\sigma^r(x)) = \begin{cases}
\sigma^r(\partial^{[s-r]}(x)) & \text{if $r \le s$} \\
0 & \text{if $r>s$} \end{cases}.
\end{displaymath}
\end{proposition}

\begin{proof}
We can check after applying $\phi^r$ to each side, since $\phi$ is injective. The result now follows from Proposition~\ref{prop:hasse1}.
\end{proof}

\begin{corollary} \label{cor:levelt}
  Let $x \in R$ be an element of level $\ge t$. Then:
  \[
    \partial^{[t]}(x^n) = nx^{n-1} \partial^{[t]}(x).
  \]
\end{corollary}

\begin{proof}
  Write $x = \sigma^t(y)$ for some $y \in R$. By Proposition~\ref{prop:hassesigma}, we have
  \[
    \partial^{[t]}(\sigma^t(y^n)) = \sigma^t(\partial^{[0]}(y^n)) = \sigma^t(n y^{n-1} \partial^{[0]}(y)) = n x^{n-1} \partial^{[t]}(\sigma^t(y)). \qedhere
  \]
\end{proof}

\begin{proposition} \label{prop:kerhasse1}
We have $\bigcap_{\partial \in \cD} \ker(\partial^{[0]}) = \im(\sigma)$.
\end{proposition}

\begin{proof}
We have $\im(\sigma) \subset \ker(\partial^{[0]})$ for any $\partial \in \cD$ by Proposition~\ref{prop:hassesigma}. Conversely, suppose that $\partial^{[0]}(x)=0$ for all $\partial \in \cD$. By (D3) with $s=0$, we find $x \in \im(\sigma)$.
\end{proof}

\begin{proposition} \label{prop:reduced}
The ring $R$ is reduced.
\end{proposition}

\begin{proof}
We must show that $\sigma$ is injective (Proposition~\ref{prop:pfac-reduced}). Suppose not, and let $x \in R$ be a non-zero homogeneous element of minimal degree with $\sigma^k(x)=0$ for some $k$. Then $x \not\in \im(\sigma)$: indeed, if $x=\sigma(y)$ then $y$ would have lower degree and $\sigma^{k+1}(y)=0$, contradicting the choice of $x$. Since $x \not\in \im(\sigma)$, there exists $\partial \in \cD$ with $\partial^{[0]}(x) \ne 0$ by Proposition~\ref{prop:kerhasse1}. But then $0=\partial^{[k]}(\sigma^k(x)) = \sigma^k(\partial^{[0]}{x})$ (Proposition~\ref{prop:hassesigma}), and $\partial^{[0]}(x)$ has lower degree than $x$, again contradicting the choice of $x$. It follows that $\sigma$ is injective.
\end{proof}

\begin{proposition} \label{prop:hasselev}
Let $x \in R$. The following are equivalent:
\begin{enumerate}[\indent \rm (a)]
\item $x$ has level $\ge t$.
\item $\partial^{[r]}(x)=0$ for all $0 \le r < t$ and all $\partial \in \cD$.
\end{enumerate}
\end{proposition}

\begin{proof}
Proposition~\ref{prop:hassesigma} shows that (a) implies (b). Conversely, suppose that (b) holds, and say that $x$ has level $s<t$. Write $x=\sigma^s(y)$. Then $0=\partial^{[s]}(x)=\sigma^s(\partial^{[0]}(y))$ for all $\partial \in \cD$. Since $\sigma$ is injective (Propositions~\ref{prop:pfac-reduced} and~\ref{prop:reduced}), we have $\partial^{[0]}(y)=0$. Since this holds for all $\partial$, we have $y=\sigma(z)$ for some $z$ by Proposition~\ref{prop:kerhasse1}, and so $x=\sigma^{s+1}(z)$, contradicting the assumption that $x$ has level $s$. Thus $x$ has level $\ge t$.
\end{proof}

\begin{corollary} \label{cor:partial-allr}
If $x \in R$ is homogeneous of positive degree and $\partial^{[r]}(x)=0$ for all $r \ge 0$ and all $\partial \in \cD$, then $x=0$.
\end{corollary}

\begin{proof}
By Proposition~\ref{prop:hasselev}, $x$ has level $\ge t$ for all $t$, and so $x=0$. (Recall that $\cF^t R_d=0$ for $t \gg 0$ and $d>0$.)
\end{proof}

\begin{proposition} \label{prop:levelder}
Let $\epsilon_1, \ldots, \epsilon_s \in R_0$. Suppose that $x \in R$ has level $\ge r$, with $r \ge 1$, and $\partial^{[r]}(x) \in \sum_{j=1}^s \epsilon_j R^p$ for all $\partial \in \cD$. Then $x \in \im(\sigma^{r+1}) + \sum_{j=1}^s \epsilon_j \im(\sigma^r \phi)$. In particular, $x \in \im(\sigma^{r+1})+\sum_{j=1}^s \epsilon_j R^p$.
\end{proposition}

\begin{proof}
Write $x=\sigma^r(y)$. Let $\partial \in \cD$ and write $\partial^{[r]}(x)=\sum_{j=1}^s \epsilon_j z_i^p$. Then
\begin{displaymath}
\sigma^r(\partial^{[0]}(y)) = \partial^{[r]}(x) = \sum_{j=1}^s \epsilon_j \sigma(\phi(z_i)).
\end{displaymath}
Since $\sigma$ is injective (Proposition~\ref{prop:pfac-reduced} and~\ref{prop:reduced}), we find $\sigma^{r-1}(\partial^{[0]}(y)) \in \sum_{j=1}^s \epsilon_j \im(\phi)$, from which it follows (by Proposition~\ref{prop:pfac2}) that $\partial^{[0]}(y) \in \sum_{j=1}^s \epsilon_j \im(\phi)$. Since this holds for all $\partial \in \cD$, we see from (D3) that $y \in \im(\sigma) + \sum_{j=1}^s \epsilon_j \im(\phi)$. Applying $\sigma^r$ yields the stated result.
\end{proof}

\section{The polynomiality theorem}

\subsection{Main theorem and initial reductions}

The following is the main theorem of this section:

\begin{theorem} \label{thm:poly}
Let $R$ be a graded ring with $R_0$ a field of characteristic~$p$. Then $R$ is a polynomial ring if and only if it admits an $F$-factorization with an admissible set of Hasse derivations.
\end{theorem}

The rest of this section is devoted to the proof of the theorem. From Examples~\ref{ex:pfac} and~\ref{ex:hasse}, we see that a polynomial ring admits an $F$-factorization with an admissible set of Hasse derivations. We must prove the converse. We will do this via a series of reductions and induction arguments. Below we give two statements $\sA_\cE$ and $\sI$; their connections to the main theorem are as follows:
\[
  \sI \stackrel{\rm Lemma~\ref{lem:I=>A}}{\Longrightarrow} \sA_{\cE} \stackrel{\rm Lemma~\ref{lem:A=>Thm}}{\Longrightarrow} {\rm Theorem~\ref{thm:poly}}.
\]
The proof of $\sI$ is handled in \S\ref{sec:proof-I}.

For the rest of this section, we fix the ring $R$, an $F$-factorization $(\phi, \sigma)$ on $R$, and an admissible set $\cD$ of Hasse derivations on $R$. We put $\bk=R_0$ and write $R_+$ for the homogeneous maximal ideal of $R$.

Recall the level filtration $\cF^{\bullet} R$ on $R$ (Definition~\ref{defn:level}). We define $\cF^r(R_+/R_+^2)$ to be the image of $\cF^rR_+$. By definition, every level $r$ element of $R_+/R_+^2$ admits a level $r$ lift to $R_+$. Let $\ol{\cE}$ be a basis of $R_+/R_+^2$ consisting of homogeneous elements that is compatible with the level filtration. We write $\ol{\cE}_{d,r}$ for the set of degree $d$ level $r$ elements in $\ol{\cE}$. The compatibility with the filtration means that $\bigcup_{s \ge r} \ol{\cE}_{d,s}$ forms a basis of $\cF^r (R_+/R_+^2)_d$ for all $r$ and $d$. We let $\cE_{d,r}$ be a set of degree $d$ level $r$ elements of $R$ mapping bijectively to $\ol{\cE}_{d,r}$ under the reduction map, and we let $\cE$ be the union of these sets. Since $\cE$ lifts a basis of $R_+/R_+^2$, it generates $R$ as a $\bk$-algebra. Our goal is to show that $\cE$ is algebraically independent.

For a subset $E$ of $\cE$, consider the following statement:
\begin{itemize}
\item[$\sA_E$:] Let $x_1, \ldots, x_n \in E$ be distinct elements, and let $\epsilon_1, \ldots, \epsilon_s \in \bk$ (we allow $s=0$). Suppose that $f(x_1, \ldots, x_n) \in \sum_{j=1}^s \epsilon_j R^p$ for some polynomial $f \in \bk[X_1, \ldots, X_n]$. Then $f \in \sum_{j=1}^s \epsilon_j \bk[X_1, \ldots, X_n]^p$.
\end{itemize}
It suffices to prove this statement for $E=\cE$:

\begin{lemma} \label{lem:A=>Thm}
If $\sA_{\cE}$ holds then $\cE$ is algebraically independent.
\end{lemma}

\begin{proof}
Suppose that $f(x_1, \ldots, x_n)=0$ is an algebraic relation between distinct elements of $\cE$. By $\sA_{\cE}$ with $s=0$, we find $f=0$. Thus $\cE$ is algebraically independent.
\end{proof}

We will prove $\sA_E$ for all $E$ by induction on $E$. We have to proceed somewhat carefully in this induction. The precise inductive statement is the following:
\begin{itemize}
\item[$\sI$:] Fix $d$ and $t$ and let $E$ be a subset of $\cE$ satisfying the following:
\begin{displaymath}
\cE_{<d, \bullet} \cup \cE_{d, >t} \subset E \subset \cE_{<d, \bullet} \cup \cE_{d, \ge t}
\end{displaymath}
Let $x \in \cE_{d,t} \setminus E$, and let $E'=E \cup \{x\}$. Then $\sA_E$ implies $\sA_{E'}$.
\end{itemize}

\begin{lemma} \label{lem:I=>A}
If $\sI$ holds then $\sA_{\cE}$ holds.
\end{lemma}

\begin{proof}
In this proof, we write $\sA(E)$ in place of $\sA_E$ for readability. Note that if $E$ is a directed union of some family of subsets $\{E_i\}_{i \in I}$ then $\sA(E)$ holds if and only if $\sA(E_i)$ holds for all $i \in I$. We use this several times.

We prove that $\sA(\cE_{\le d, \bullet})$ holds for all $d$ by induction on $d$. Thus suppose that $\sA(\cE_{<d,\bullet})$ holds, and let us show that $\sA(\cE_{\le d,\bullet})$ holds. To do this, we show that $\sA(\cE_{<d,\bullet} \cup \cE_{d,\ge t})$ holds for all $t$, by descending induction on $t$. For $t \gg 0$, the statement is vacuous since $\cE_{d,\ge t}$ is empty. Thus assume that $\sA(\cE_{<d,\bullet} \cup \cE_{d,>t})$ holds, and let us show that $\sA(\cE_{<d,\bullet} \cup \cE_{d,\ge t})$ holds. To do this, it suffices to show that $\sA(\cE_{<d,\bullet} \cup \cE_{d,>t} \cup \{e_1,\dots,e_r\})$ holds for all finite subsets $\{e_1,\dots,e_r\}$ of $\cE_{d,t}$. This follows inductively by applying $\sI$ with $x = e_i$ and $E = \cE_{<d,\bullet} \cup \cE_{d,>t} \cup \{e_1,\dots,e_{i-1}\}$ for $i=1,\dots,r$.
\end{proof}

\subsection{Proof of $\sI$} \label{sec:proof-I}
We now prove statement $\sI$. Fix $d$, $t$, $E$, $x$, and $E'$ as in $\sI$, and suppose $\sA_E$ holds. We will prove $\sA_{E'}$. For this, we will introduce further auxiliary statements $\sB_{n,m}$, $\sC$ and $\sC'$ below. The logical implications between these statements is as follows:
\[
\text{$\sA_E$ and $\sB_{n,m}$ for all $n$ and $m$} \stackrel{\rm Lemma~\ref{lem:B=>AE}}{\Longrightarrow} \sA_{E'}.
\]
\[
\text{($\sB_{k,\ell}$ for all $k<n$ and all $\ell$) and ($\sB_{n,k}$ for all $k<m$) and $\sC$} \Longrightarrow \sB_{n,m}
\]
\[
{\rm Lemma~\ref{lem:C'}}\Longrightarrow  \sC', \qquad \sC' \stackrel{{\rm Lemma~\ref{lem:C'=>C}}}{\Longrightarrow} \sC,
\]
In words, the first statement shows that we can deduce $\sA_{E'}$ from $\sB_{n,m}$ for all $n$ and $m$, since we have assumed $\sA_E$ holds. The second statement shows that we can prove $\sB_{n,m}$ using induction on $(n,m)$, once we have $\sC$. And the final statement shows that $\sC$ is implied by $\sC'$, which, in turn, is established by Lemma~\ref{lem:C'}.

We now start on the details. We let $A \subset R$ be the $\bk$-subalgebra generated by $E$.

\begin{lemma} \label{lem:B}
The element $x$ does not belong to $A$.
\end{lemma}

\begin{proof}
Suppose $x$ could be expressed as a polynomial in the elements in $E$. Reducing this expression modulo $R_+^2$, we would find that $x$ belongs to the $\bk$-span of $E$. But this contradicts the linear independence of the set $\cE$.
\end{proof}

\begin{lemma} \label{lem:A}
Let $y \in R$ be homogeneous. Suppose $y$ has degree $<d$, or degree $d$ and level $>t$. Then $y \in A$.
\end{lemma}

\begin{proof}
First suppose $y$ has degree $e<d$. Then the image of $y$ in $R_+/R_+^2$ is a linear combination of elements of $\ol{\cE}_{e,\bullet}$. Thus if $y'$ is the corresponding linear combination of elements of $\cE_{e,\bullet}$ then $y-y' \in R_+^2$, that is, we can write $y-y'=\sum_{i=1}^n a_i b_i$ with $a_i$ and $b_i$ of degree $<e$. By induction on $e$, the elements $a_i$ and $b_i$ belong to $A$, while $y'$ belongs to $A$ by definition. Thus $y \in A$, as required.

Now suppose that $\deg(y)=d$ and $y$ has level $>t$. Then the image of $y$ in $R_+/R_+^2$ is a linear combination of elements of $\ol{\cE}_{d,>t}$. Again, taking $y'$ to be the corresponding linear combination of elements of $\cE_{d,>t}$, we see that $y-y' \in R_+^2$. Thus, the same reasoning shows that $y \in A$.
\end{proof}

By Lemma~\ref{lem:A}, we see that $\partial^k(x) \in A$ for any $\partial \in \cD$ and $k>0$, as $\partial^k(x)$ has degree $<d$. It follows from Proposition~\ref{prop:hassepower} that for $k \ge 1$, $n \ge 1$, and $\partial \in \cD$, the element $\partial^k(x^n)$ can be expressed as a polynomial in $x$ of degree $\le n-1$ with coefficients in $A$. We use this constantly and without reference in what follows.

Consider the following statement:
\begin{itemize}
\item[$\sB_{n,m}$:] Let $a_0, \ldots, a_n \in A$ with $\deg(a_n) \le m$ and let $\epsilon_1, \ldots, \epsilon_s \in \bk$. Suppose that $\sum_{i=0}^n a_i x^i \in \sum_{j=1}^s \epsilon_j R^p$. Then $a_i \in \sum_{j=1}^s \epsilon_j R^p$ for all $i$, and $a_i=0$ if $p \nmid i$.
\end{itemize}
We note that in the conclusion it would be equivalent to state that $a_i \in \sum_{j=1}^s \epsilon_j A^p$, since $\sA_E$ holds.

\begin{lemma} \label{lem:B=>AE}
Suppose $\sB_{n,m}$ holds for all $n$ and $m$. Then $\sA_{E'}$ holds.
\end{lemma}

\begin{proof}
Let $x_1, \ldots, x_k$ be distinct elements of $E'$ and suppose $f(x_1, \ldots, x_k) \in \sum_{j=1}^s \epsilon_j R^p$. Without loss of generality, we may assume $x_k=x$ and $x_1, \ldots, x_{k-1} \in E$. Write $f(X_1, \ldots, X_k) = \sum_{i=0}^n g_i(X_1, \ldots, X_{k-1}) X_k^i$ for polynomials $g_i$. Then $\sum_{i=0}^n g_i(x_1, \ldots, x_{k-1}) x^i \in \sum_{j=1}^s \epsilon_j R^p$. By $\sB_{n,\bullet}$, we see that $g_i(x_1, \ldots, x_{k-1}) \in \sum_{j=1}^s \epsilon_j R^p$ for all $i$ and $g_i(x_1, \ldots, x_{k-1})=0$ for $p \nmid i$. By $\sA_E$, we see that $g_i \in \sum_{j=1}^s \epsilon_j \bk[X_1, \ldots, X_{k-1}]^p$ for all $i$, and (by taking $s=0$) $g_i=0$ for $p \nmid i$. It follows that $f \in \sum_{j=1}^s \epsilon_j \bk[X_1, \ldots, X_k]^p$, and so $\sA_{E'}$ holds.
\end{proof}

We now prove the statement $\sB_{n,m}$ for all $n$ and $m$. We do this by induction on $n$ and $m$. It is clear that $\sB_{0,m}$ holds for all $m$. We now fix $n>0$ and $m \ge 0$ and assume that $\sB_{<n,\bullet}$ and $\sB_{n,<m}$ hold, and we will prove that $\sB_{n,m}$ holds. To this end, fix $a_0, \ldots, a_n \in A$ with $\deg(a_n) \le m$ and $\epsilon_1, \ldots, \epsilon_s\in \bk$ such that
\begin{equation} \label{eq1}
\sum_{i=0}^n a_i x^i \in \sum_{j=1}^s \epsilon_j R^p.
\end{equation}
We must prove:
\begin{itemize}
\item[($\sC$)] We have $a_i \in \sum_{j=1}^s \epsilon_j R^p$ for all $i$, and $a_i=0$ for $p \nmid i$.
\end{itemize}
Clearly, if we do this then we will have established $\sB_{n,m}$, and the theorem will follow. Without loss of generality, we assume $\epsilon_1, \ldots, \epsilon_s$ to be linearly independent over $\bk^p$. Before tackling $\sC$, we make one simple reduction. Consider the following statement:
\begin{itemize}
\item[($\sC'$)] We have $a_n \in \sum_{j=1}^s \epsilon_j R^p$ and $a_n=0$ if $p \nmid n$.
\end{itemize}
The following lemma shows it suffices to prove this statement:

\begin{lemma} \label{lem:C'=>C}
If $\sC'$ holds then $\sC$ holds.
\end{lemma}

\begin{proof}
Suppose $\sC'$ holds. Combining this with \eqref{eq1}, we see that $\sum_{i=0}^{n-1} a_i x^i$ belongs to $\sum_{j=1}^s \epsilon_j R^p$. Thus by $\sB_{<n,\bullet}$, we see that $a_i$ belongs to $\sum_{j=1}^s \epsilon_j R^p$ and $a_i=0$ for $p \nmid i$, at least for $0 \le i \le n-1$. Of course, we know this for $i=n$ as well by $\sC'$. Thus $\sC$ holds.
\end{proof}

Before establishing $\sC'$, we need one auxiliary lemma:

\begin{lemma} \label{lem:depkp}
Suppose that we have an equation $\sum_{j=1}^N \delta_j y_j=0$ where the $\delta_j$ belong to $\bk$, the $y_j$ belong to $R^p$, and not all $y_j$ vanish. Then the $\delta_j$ are linearly dependent over $\bk^p$.
\end{lemma}

\begin{proof}
We can assume the $y_j$ all have the same degree, and we proceed by induction on the degree. If the $y_j$ have degree~0 then the statement is clear. Otherwise, we can find $\partial \in \cD$ such that $\partial^{[r]}(y_j) \ne 0$ for some $r \ge 1$ and some $j$ (Corollary~\ref{cor:partial-allr}). Applying this to the given equation yields $\sum_{j=1}^N \delta_j \partial^{[r]}(y_j)=0$. Each $\partial^{[r]}(y_j)$ is a $p$th power (Proposition~\ref{prop:hasse1}), and not all vanish. Since the degree has dropped, the result follows by induction.
\end{proof}

We now come to the heart of the argument:

\begin{lemma} \label{lem:C'}
Statement $\sC'$ holds.
\end{lemma}

\begin{proof}
We proceed in five cases. In what follows, in expressions such as $a x^{n-1} + \cdots$, the $\cdots$ will always refer to an $A$-linear combination of lower powers of $x$.

\vskip.5\baselineskip
\textit{\textbf{Case 1:} $p$ does not divide $n$ or $n-1$, and $m=0$.}
Since $m=0$, the leading coefficient $a_n$ is constant. We must show that it vanishes. Recall that $t$ is the level of $x$. We first claim that $a_{n-1}$ has level $\ge t$. Let $0 \le r < t$. Applying $\partial^{[r]}$ to \eqref{eq1} for some $\partial \in \cD$, and using the fact that $\partial^{[r]}(x)=0$ (Proposition~\ref{prop:hasselev}), and thus that $\partial^{[r]}(x^n)=0$ (Corollary~\ref{cor:levelt}), we see that $\partial^{[r]}(a_{n-1}) x^{n-1} + \cdots \in \sum_{j=1}^s \epsilon_j R^p$. By $\sB_{n-1,\bullet}$ we see that $\partial^{[r]}(a_{n-1})=0$ (here we are using $p \nmid n-1$). Since this holds for all $\partial \in \cD$, we conclude that $a_{n-1}$ has level $\ge t$ (Proposition~\ref{prop:hasselev}).

Now apply $\partial^{[t]}$ to \eqref{eq1}, for some $\partial \in \cD$. Using the identity $\partial^{[t]}(x^n) = nx^{n-1} \partial^{[t]}(x)$ (Corollary~\ref{cor:levelt}), we see that the coefficient of $x^{n-1}$ in the result is $\partial^{[t]}(na_n x+a_{n-1})$. We thus have
\begin{displaymath}
\partial^{[t]}(na_n x+a_{n-1}) x^{n-1} + \cdots \in \sum_{j=1}^s \epsilon_j R^p.
\end{displaymath}
Applying $\sB_{n-1,m}$, and using $p \nmid n-1$, we see that $\partial^{[t]}(na_n x+a_{n-1})=0$. Since we already saw that $na_n x+a_{n-1}$ has level $\ge t$, this shows that $na_n x+a_{n-1}$ has level $\ge t+1$ (Proposition~\ref{prop:hasselev}). Since $na_n x+a_{n-1}$ has the same degree $d$ as $x$ but greater level, it belongs to $A$ (Lemma~\ref{lem:A}). Since $a_{n-1}$ also belongs to $A$ and $n \ne 0$ in $\bk$, we see that $a_n x$ belongs to $A$. Since $x$ does not belong to $A$ (Lemma~\ref{lem:B}), we see that $a_n=0$, which completes the proof.

\vskip.5\baselineskip
\textit{\textbf{Case 2:} $p$ does not divide $n$, $p$ does divide $n-1$, and $m=0$.}
The argument is similar to the previous case. We first treat the case $t=0$. Applying $\partial^{[0]}$ to \eqref{eq1}, for some $\partial \in \cD$, we find
\begin{displaymath}
\partial^{[0]}(na_nx+a_{n-1}) x^{n-1}+\cdots = 0.
\end{displaymath}
By $\sB_{n-1,\bullet}$ with $s=0$, we see that $\partial^{[0]}(na_nx+a_{n-1})=0$. Since this holds for all $\partial \in \cD$, we see that $na_nx+a_{n-1}$ has level $\ge 1$. As in Case 1, this implies $na_nx+a_{n-1} \in A$, from which we conclude $a_n=0$.

Now suppose $t>0$. We first claim that there is some $b \in \sum_{j=1}^s \epsilon_j R^p$ such that $a_{n-1}+b$ is homogeneous and has level $\ge t$. To begin, applying $\partial^{[0]}$ to \eqref{eq1} for some $\partial \in \cD$, we find $\partial^{[0]}(a_{n-1}) x^{n-1}+\cdots =0$, whence $\partial^{[0]}(a_{n-1})=0$ by $\sB_{n-1,\bullet}$ (with $s=0$). Since this holds for all $\partial \in \cD$, we see that $a_{n-1}$ has level $\ge 1$. Now suppose $1 \le k < t$ and we can find $b \in \sum_{j=1}^s \epsilon_j R^p$ so that $a_{n-1}+b$ is homogeneous and has level $\ge k$. We have just shown this to be the case for $k=1$, by taking $b=0$. Applying $\partial^{[k]}$ to \eqref{eq1} for some $\partial \in \cD$, we find $\partial^{[k]}(a_{n-1}) x^{n-1} + \cdots \in \sum_{j=1}^s \epsilon_j R^p$. By $\sB_{n-1,\bullet}$, we see that $\partial^{[k]}(a_{n-1}) \in \sum_{j=1}^s \epsilon_j R^p$. Thus we have $\partial^{[k]}(a_{n-1}+b) \in \sum_{j=1}^s \epsilon_j R^p$ as well. Since $a_{n-1}+b$ has level $\ge k$ and this holds for all $\partial \in \cD$, Proposition~\ref{prop:levelder} implies $a_{n-1}+b = \sigma^{k+1}(c)+b'$ for some $b' \in \sum_{j=1}^s \epsilon_j R^p$. Thus $a_{n-1}+(b-b')$ has level $\ge k+1$. Replacing $b$ with $b-b'$ and continuing, the claim follows.

Fix $b \in \sum_{j=1}^s \epsilon_j R^p$ so that $a_{n-1}+b$ has level $\ge t$. Apply $\partial^{[t]}$ to \eqref{eq1}, for some $\partial \in \cD$. We find
\begin{displaymath}
\partial^{[t]}(na_nx+a_{n-1}) x^{n-1} + \cdots \in \sum_{j=1}^s \epsilon_j R^p.
\end{displaymath}
By $\sB_{n-1,\bullet}$, we find $\partial^{[t]}(na_nx+a_{n-1}) \in \sum_{j=1}^s \epsilon_j R^p$. Of course, we also have $\partial^{[t]}(na_nx+a_{n-1}+b) \in \sum_{j=1}^s \epsilon_j R^p$ by Corollary~\ref{cor:partial-preserve} because $b \in \sum_{j=1}^s \epsilon_j R^p$. Since this holds for all $\partial \in \cD$ and $na_nx+a_{n-1}+b$ has level $\ge t$ with $t>0$, Proposition~\ref{prop:levelder} implies $na_nx+a_{n-1}+b=\sigma^{t+1}(y)+b'$ for some $b' \in \sum_{j=1}^s \epsilon_j R^p$, and we may further assume that $\sigma^{t+1}(y)$ and $b'$ are homogeneous of the same degree. Since $\sigma^{t+1}(y)$ has the same degree as $x$ but higher level, it belongs to $A$. Since $b$ and $b'$ have the same degree as $x$ and belong to $R_+^2$, they also belong to $A$. Thus we find $a_n x \in A$, and so $a_n=0$.

\vskip.5\baselineskip
\textit{\textbf{Case 3:} $p$ does not divide $n$ and $m>0$.}
If $\deg(a_n)=0$ then the result follows from $\sB_{n,<m}$, so we assume $\deg(a_n)>0$. Applying $\partial^{[r]}$ to \eqref{eq1}, for some $\partial \in \cD$ and $r \ge 0$, we find
\begin{displaymath}
\partial^{[r]}(a_n) x^n + \cdots \in \sum_{j=1}^s \epsilon_j R^p.
\end{displaymath}
By $\sB_{n,m-1}$ and the fact that $p \nmid n$, we conclude that $\partial^{[r]}(a_n)=0$. Since this holds for all $\partial \in \cD$ and all $r \ge 0$, we conclude that $a_n=0$ by Corollary~\ref{cor:partial-allr}.

\vskip.5\baselineskip
\textit{\textbf{Case 4:} $p$ divides $n$ and $m=0$.}
Let $\epsilon_{s+1}=a_n$, which belongs to $\bk$ since $m=0$. Moving the $a_n x^n$ term in \eqref{eq1} to the other side, and noting that $x^n \in R^p$ as $p \mid n$, we find
\begin{displaymath}
\sum_{i=0}^{n-1} a_i x^i \in \sum_{j=1}^{s+1} \epsilon_j R^p.
\end{displaymath}
Applying $\sB_{<n,\bullet}$, it follows that $a_i \in \sum_{j=1}^{s+1} \epsilon_j R^p$ for each $0 \le i \le n$, and $a_i=0$ for $p \nmid i$. By $\sA_E$, we in fact have $a_i \in \sum_{j=1}^{s+1} \epsilon_j A^p$. Let $b_i \in A$ be such that $a_i \in \epsilon_{s+1} b_i^p + \sum_{j=1}^s \epsilon_j R^p$. Take $b_i=0$ for $p \nmid i$, and also put $b_n=1$. Putting this into \eqref{eq1}, we find $\epsilon_{s+1} \sum_{i=0}^n b_i^p x^i \in \sum_{j=1}^s \epsilon_j R^p$. The summation on the left is $(\sum_{i=0}^n b_{i/p} x^{i/p})^p$, where the sum is over indices divisible by $p$. This sum is non-zero by $\sB_{n/p,\bullet}$ since $b_n=1$. Thus, by Lemma~\ref{lem:depkp}, and our assumption that $\{\epsilon_1,\dots,\epsilon_s\}$ is linearly independent over $\bk^p$, we see that $a_n=\epsilon_{s+1}$ is in the $\bk^p$-span of $\epsilon_1, \ldots, \epsilon_s$, and this establishes $\sC'$.

\vskip.5\baselineskip
\textit{\textbf{Case 5:} $p$ divides $n$ and $m>0$.}
If $\deg(a_n)=0$ then the result follows from $\sB_{n,<m}$, so assume $\deg(a_n)>0$. We claim that for any $r$ we can find $b \in \sum_{j=1}^s \epsilon_j R^p$ such that $a_n+b$ has level $\ge r$.

To start, apply $\partial^{[0]}$ to \eqref{eq1} for some $\partial \in \cD$. We find $\partial^{[0]}(a_n) x^n + \cdots =0$. From $\sB_{n,<m}$, it follows that $\partial^{[0]}(a_n)=0$. Since this holds for all $\partial \in \cD$, we conclude that $a_n$ has level $\ge 1$. This proves the claim for $r=1$ (take $b=0$).

Suppose now that we have found $b$ such that $a_n+b$ has level $\ge r$, with $r \ge 1$. Applying $\partial^{[r]}$ to \eqref{eq1} for some $\partial \in \cD$, we find $\partial^{[r]}(a_n) x^n + \cdots \in \sum_{j=1}^s \epsilon_j R^p$. By $\sB_{n,<m}$, we have $\partial^{[r]}(a_n) \in \sum_{j=1}^s \epsilon_j R^p$. Of course, we also have $\partial^{[r]}(a_n+b) \in \sum_{j=1}^s \epsilon_j R^p$. As this holds for all $\partial$ and $a_n+b$ has level $\ge r$, Proposition~\ref{prop:levelder} implies that $a_n+b=\sigma^{r+1}(c)+b'$ for some $b' \in \sum_{j=1}^s \epsilon_j R^p$. Thus $a_n+b-b'$ has level $\ge r+1$. The claim follows.

Take $r \gg m$, and pick $b \in \sum_{j=1}^s \epsilon_j R^p$ so that $a_n+b$ has level $\ge r$. Then $a_n+b=0$, since the level of a non-zero element of positive degree is bounded in terms of its degree. Thus $a_n \in \sum_{j=1}^s \epsilon_j R^p$, which establishes $\sC'$.
\end{proof}

\section{Applications} \label{s:app}

\subsection{Inverse limit rings}

For a (non-graded) ring $A$ and an infinite set $\cU$, we put
\begin{displaymath}
A \invl x_i \invr_{i \in \cU} = \varprojlim A[x_i]_{i \in \cV},
\end{displaymath}
where the inverse limit is taken in the category of graded rings over all finite subsets $\cV$ of $\cU$, and $A[x_i]_{i \in \cV}$ denotes the standard-graded polynomial ring in the indicated variables. Thus $A \invl x_i \invr_{i \in \cU}$ is a graded ring, and a degree $d$ element can be written uniquely in the form $\sum c_e x^e$, where the sum is over multi-indices $e$ of degree $d$ and the $c_e$ are arbitrary elements of $A$. (A multi-index is a function $e \colon \cU \to \bN$ of finite support; its degree is the sum of its values.) Fix a field $K$ and an infinite set $\cU$, and let $\bR=K \invl x_i \invr_{i \in \cU}$. The following theorem is a slight generalization of Theorem~\ref{mainthm}:

\begin{theorem} \label{thm:inverse}
The ring $\bR$ is a polynomial $K$-algebra.
\end{theorem}

\begin{proof}
If $K$ has characteristic~0, this follows from \cite[Theorem~5.3]{stillman} (which also covers some cases in positive characteristic). We note that the statement of \cite[Theorem~5.3]{stillman} takes $\cU$ to be countable, but the argument applies generally. Suppose now that $K$ has positive characteristic~$p$. Define endomorphisms $\phi$ and $\sigma$ of $\bR$ by
\begin{displaymath}
\phi \left( \sum c_e x^e \right) = \sum c_e^p x^e, \qquad
\sigma \left( \sum c_e x^e \right) = \sum c_e x^{pe},
\end{displaymath}
Then $(\phi, \sigma)$ defines an $F$-factorization on $\bR$: the argument in Proposition~\ref{prop:standard-fact} applies with essentially no changes.

Let $\ol{\partial}_i$ be the Hasse derivative with respect to $x_i$ on the polynomial ring $K[x_i]_{i \in \cU}$. Let $\partial^n_i$ be the unique continuous extension of $\ol{\partial}^n_i$ to $\bR$; concretely,
\begin{displaymath}
\partial_i^n \left( \sum c_e x^e \right) = \sum c_i \ol{\partial}_i^n(x^e).
\end{displaymath}
One easily sees that $\{\partial_i^n\}_{n \ge 0}$ is a Hasse derivation $\partial_i$ on $\bR$ that is homogeneous of degree $-1$. Let $\cD$ be the set of these Hasse derivations over all $i$. An argument as in Example~\ref{ex:hasse} shows that the set $\cD$ is admissible, and the result then follows from Theorem~\ref{thm:poly}.
\end{proof}

\subsection{Bounded inverse limit rings} \label{ss:bounded}

We now prove a generalization of Theorem~\ref{thm:inverse}. Let $K$, $\cU$, and $\bR$ be as in the previous section. Fix a subring $A$ of $K$ with $\Frac(A)=K$. We say that a subset $S$ of $K$ is {\bf bounded} if there exists a non-zero element $b \in A$ such that $S \subset b^{-1} A$. We let $\bR^{\flat} \subset \bR$ be the subring consisting of elements with bounded coefficients (i.e., the set of coefficients forms a bounded subset of $K$). We note that $\bR^{\flat}$ is naturally identified with $K \otimes_A A \invl x_i \invr_{i \in \cU}$, which is typically \emph{not} isomorphic to $K \invl x_i \invr_{i \in \cU}$. The interest in this ring comes from \cite[\S 5]{stillman}, where it plays a key role in our second proof of Stillman's conjecture: it arises when localizing the ring $A \invl x_i \invr_{i \in \cU}$ over the generic point of $\Spec(A)$.

Suppose for the moment that $K$ has characteristic $p$. For a cardinal $\kappa$, we consider the following condition on $A$:
\begin{itemize}
\item[$(C_{\kappa})$] Given $\epsilon_1, \ldots, \epsilon_s \in K$ and a subset $S$ of $A \cap \sum_{j=1}^s \epsilon_j K^p$ of cardinality at most $\kappa$, there exists a non-zero element $b \in A$ such that $S \subset b^{-p} \sum_{j=1}^s \epsilon_j A^p$.
\end{itemize}
We also consider the following simpler condition:
\begin{itemize}
\item[$(C)$] Given $\epsilon_1, \ldots, \epsilon_s \in K$, there exists a non-zero element $b \in A$ such that $A \cap \sum_{j=1}^s \epsilon_j K^p \subset b^{-p} \sum_{j=1}^s \epsilon_j A^p$.
\end{itemize}
We note that $(C)$ holds if and only if $(C_{\kappa})$ holds for all $\kappa$.

\begin{theorem} \label{thm:inverse2}
Suppose that either $K$ has characteristic~$0$, or $K$ has characteristic~$p$ and condition $(C_{\kappa})$ holds for $\kappa=\# \cU$. Then $\bR^{\flat}$ is a polynomial $K$-algebra.
\end{theorem}

If $K$ has characteristic~0, the result follows from \cite[Theorem~5.3]{stillman}. We assume for the rest of the proof that $K$ has characteristic~$p$. We aim to show that the $F$-factorization and admissible set $\cD$ on $\bR$ constructed in the proof of Theorem~\ref{thm:inverse} restrict to such a structure on $\bR$. It is clear that $\phi$ and $\sigma$ map $\bR^{\flat}$ into itself; we let $\phi^{\flat}$ and $\sigma^{\flat}$ denote their restrictions to $\bR^{\flat}$. Similarly, it is clear that each Hasse derivative $\partial_i$ maps $\bR^{\flat}$ into itself; we let $\cD^{\flat}=\{ \partial_i \vert_{\bR^{\flat}} \}_{i \in \cU}$.

\begin{lemma}
The pair $(\phi^{\flat}, \sigma^{\flat})$ is an $F$-factorization on $\bR^{\flat}$.
\end{lemma}

\begin{proof}
It is clear that (F1), (F2), and (F3) hold. We now verify (F4). Thus let $\epsilon_1, \ldots, \epsilon_s \in K$ be given, and suppose $f$ belongs to $\im(\sigma^{\flat}) \cap \sum_{j=1}^s \epsilon_j \im(\phi^{\flat})$. Write $f=\sum_{j=1}^s \epsilon_j \phi(g_j)$ for some $g_j \in \bR^{\flat}$. Decompose $g_j$ as $g_j^1+g_j^2$, where $g_j^1$ contains all monomials $x^e$ where $p \mid e$, and $g_j^2$ contains the other monomials. Every monomial occurring in $f$ with non-zero coefficient has the form $x^e$ with $p \mid e$, since $f \in \im(\sigma)$, and so we see that $f=\sum_{j=1}^s \epsilon_j \phi(g_j^1)$. Since all exponents appearing in $g_j^1$ are divisible by $p$, we can write $g_j^1=\sigma(h_j)$ for some $h_j \in \bR$. The set of coefficients appearing in $h_j$ is the same as the set of coefficients appearing in $g_j^1$, and is thus bounded, and so $h_j \in \bR^{\flat}$. Thus $\phi(g_j^1)=\phi(\sigma(h_j)) \in (\bR^{\flat})^p$, and the claim follows.
\end{proof}

\begin{lemma}
The set $\cD^{\flat}$ is admissible if and only if condition $(C_{\kappa})$ holds, where $\kappa=\# \cU$.
\end{lemma}

\begin{proof}
It is clear that $\cD^{\flat}$ satisfies (D1) and (D2). We show that (D3) holds if and only if $(C_{\kappa})$ holds.

First suppose that (D3) holds. Let $\epsilon_1, \ldots, \epsilon_s \in K$ and let $S$ be a subset of $A \cap \sum_{j=1}^s \epsilon_j K^p$ of cardinality at most $\kappa$. Enumerate $S$ as $\{a_i\}_{i \in \cU}$; if the cardinality of $S$ is smaller than $\kappa$, simply let the values of $a_i$ repeat. Consider the element $f = \sum_{i \in \cU} a_i x_i$ of $\bR^{\flat}$. We have $\partial^{[0]}_i(f) = a_i$, which belongs to $\sum_{j=1}^s \epsilon_j K^p \subset \sum_{j=1}^s \epsilon_j \im(\phi^{\flat})$. Since this holds for all $i$, we conclude that $f \in \sum_{j=1}^s \epsilon_j \im(\phi^{\flat})$ by (D3); we note that the $\im(\sigma^{\flat})$ summand in (D3) is irrelevant here since $f$ is linear. Write $f=\sum_{j=1}^s \epsilon_j \phi(g_j)$ with $g_j \in \bR^{\flat}$. Letting $b \in A$ be a non-zero element such that $b g_j \in A \invl x_i \invr_{i \in \cU}$ for all $j$, we see that $a_i \in b^{-p} \sum_{j=1}^s \epsilon_j A^p$ for all $j$, and so $(C_{\kappa})$ holds.

Conversely, suppose $(C_{\kappa})$ holds. Let $\epsilon_1, \ldots, \epsilon_s \in K$ and $f \in \bR^\flat$ be given, and suppose $\partial^{[0]}(f) \in \sum_{j=1}^s \epsilon_j \im(\phi^{\flat})$ for all $\partial \in \cD$. We aim to show that $f$ belongs to $\im(\sigma^{\flat}) + \sum_{j=1}^s \epsilon_j \im(\phi^{\flat})$. By scaling $f$ appropriately, we may assume $f$ belongs to $A \invl x_i \invr_{i \in \cU}$. Write $f=\sum c_e x^e$, as usual. Since $\cD$ satisfies (D3), we can write $f=\sigma(h)+\sum_{j=1}^s \epsilon_j \phi(g_j)$ for $h,g_1,\ldots,g_s \in \bR$. We thus see that if $p \nmid e$ then $c_e$ belongs to $A \cap \sum_{j=1}^s \epsilon_j K^p$. By $(C_{\kappa})$, we thus see that there is some non-zero $b \in A$ such that each such $c_e$ is contained in $b^{-p} \sum_{j=1}^s \epsilon_j A^p$; we note that the collection of $c_e$'s has cardinality at most $\kappa$. For $p \nmid e$, write $c_e=\sum_{j=1}^s \epsilon_j (d_{j,e}/b)^p$ where $d_{j,e} \in A$. Then $f = \sigma(h')+\sum_{j=1}^s \epsilon_j \varphi(g_j')$, where $h'=\sum_{p \mid e} c_e x^{e/p}$ and $g_j' = b^{-1} \sum_{p \nmid e} d_{j,e} x^e$. Clearly, $h'$ has coefficients in $A$, and thus belongs to $\bR^{\flat}$. It is clear from the definition that the $g_j'$ belong to $\bR^{\flat}$. We have thus verified (D3).
\end{proof}

Theorem~\ref{thm:inverse2} follows from the above two lemmas: indeed, if $(C_{\kappa})$ holds, then these lemmas show that $\bR^{\flat}$ admits an $F$-factorization and an admissible set of Hasse derivations, and is thus a polynomial ring by Theorem~\ref{thm:poly}.

\subsection{More on condition $(C)$}

To apply Theorem~\ref{thm:inverse2}, we need some understanding of condition $(C)$. We therefore turn our attention to it in this section. Throughout this section, $\bk$ denotes a field of characteristic~$p$, and $\otimes$ will denote tensor product over $\bk$.

\begin{proposition} \label{prop:easyC}
Suppose that $A$ is a noetherian domain that is finite over $A^p$. Then $A$ satisfies condition $(C)$.
\end{proposition}

\begin{proof}
Let $\epsilon_1, \ldots, \epsilon_s \in K=\Frac(A)$ be given. Let $M=A \cap \sum_{j=1}^s \epsilon_j K^p$ and $N=\sum_{j=1}^s \epsilon_j A^p$. Since $K^p N = \sum_{j=1}^s \epsilon_j K^p$, we have $M \subset K^p N$. On the other hand, $M$ is an $A^p$-submodule of $A$, and thus finitely generated as an $A^p$-module (since $A$ is finite over $A^p$ and $A^p\cong A$ is noetherian). Let $x_1, \ldots, x_n$ be generators for $M$ as an $A^p$-module and let $b \in A$ be a non-zero element such that $b^p x_i \in N$ for each $1 \le i \le n$. Then $M \subset b^{-p} N$, as required.
\end{proof}

\begin{lemma}
Let $\bk \subset K_1, K_2 \subset L$ be fields of characteristic $p$ such that $K_1$ and $K_2$ are linearly disjoint over $\bk$ and $L$ is the compositum of $K_1$ and $K_2$. Let $\{\epsilon_i\}_{i \in I}$ and $\{\delta_j\}_{j \in J}$ be subsets of $K_1$ and $K_2$. Then
\begin{displaymath}
(K_1 \otimes K_2) \cap \left( \sum_{i \in I,j \in J} \epsilon_i \delta_j L^p \right) = \left( \sum_{i \in I} \epsilon_i K_1^p \right) \otimes \left( \sum_{j \in J} \delta_j K_2^p \right),
\end{displaymath}
where both sides are regarded as subsets of $L$.
\end{lemma}

\begin{proof}
Let $X$ denote the left side and $Y$ the right side. Clearly, $Y \subset X$, so we must prove the reverse inclusion. We may as well assume that $\{\epsilon_i\}_{i \in I}$ is linearly independent over $K_1^p$; extend it to a basis $\{\epsilon_i\}_{i \in I_+}$, where $I \subset I_+$. Similarly, let $\{\delta_j\}_{j \in J_+}$ be a basis for $K_2$ over $K_2^p$. We claim that $\{\epsilon_i \delta_j\}_{i \in I_+, j \in J_+}$ is linearly independent over $L^p$. To see this, suppose that $\sum_{i \in I_+, j \in J_+} c_{i,j} \epsilon_i \delta_j=0$ is a linear relation with $c_{i,j} \in L^p$. Since $L=\Frac(K_1 \otimes K_2)$, we can clear denominators and assume that $c_{i,j} \in K_1^p \otimes K_2^p$. But $K_1 \otimes K_2$ is a free module over $K_1^p \otimes K_2^p$ with basis $\{\epsilon_i \delta_j\}_{i \in I_+, j \in J_+}$, and so $c_{i,j}=0$ for all $i$ and $j$.

Now let $x \in X$ be given. Since $x$ belongs to $K_1 \otimes K_2$, we can express it as $x=\sum_{i \in I_+, j \in J_+} c_{i,j} \epsilon_i \delta_j$ with $c_{i,j} \in K_1^p \otimes K_2^p$. On the other hand, since $x$ belongs to $\sum_{i \in I, j \in J} \epsilon_i \delta_j L^p$, we can express it as $x=\sum_{i \in I, j \in J} c'_{i,j} \epsilon_i \delta_j$ with $c'_{i,j} \in L^p$. By the previous paragraph, we must have $c'_{i,j}=c_{i,j} \in K_1^p \otimes K_2^p$, which shows that $x$ belongs to $Y$.
\end{proof}

\begin{lemma}
Let $\bk$ be a field of characteristic~$p$ and let $B_1$ and $B_2$ be integral domains containing $\bk$ such that $C=B_1 \otimes_{\bk} B_2$ is a domain. Let $K_i=\Frac(B_i)$ and $L=\Frac(C)$. Let $\epsilon_1, \ldots, \epsilon_s \in K_1$. Then
\begin{displaymath}
C \cap \left( \sum_{j=1}^s \epsilon_j L^p \right) = \left( B_1 \cap \sum_{j=1}^s \epsilon_j K_1^p \right) \otimes \left( B_2 \cap K_2^p \right),
\end{displaymath}
where each side is regarded as a subset of $C$.
\end{lemma}

\begin{proof}
Applying the previous lemma with $\{\delta_j\}_{j \in J}$ the singleton set $\{1\}$, we find
\begin{displaymath}
(K_1 \otimes K_2) \cap \sum_{j=1}^s \epsilon_j L^p = \left( \sum_{j=1}^s \epsilon_j K_1^p \right) \otimes K_2^p.
\end{displaymath}
Now intersect each side with $C$. Note that $C$ is contained in $K_1 \otimes K_2$, so the left side is just $C \cap \sum_{j=1}^s \epsilon_j L^p$. Since $C=B_1 \otimes B_2$, the right side factors as required.
\end{proof}

Recall that if $\{A_i\}_{i \in \cI}$ is a family of $\bk$-algebras then $\bigotimes_{i \in \cI} A_i$ is defined as the direct limit of the algebras $\bigotimes_{i \in \cJ} A_i$ taken over all finite subsets $\cJ$ of $\cI$.

\begin{proposition} \label{prop:tensor-C}
Let $\bk$ be a field of characteristic~$p$, let $\{A_i\}_{i \in \cI}$ be a family of $\bk$-algebras, and let $A=\bigotimes_{i \in \cI} A_i$, which we assume to be a domain. Assume:
\begin{enumerate}[\rm \indent (a)]
\item For each finite subset $\cJ$ of $\cI$, the ring $A_{\cJ} = \bigotimes_{i \in \cJ} A_i$ satisfies condition $(C)$.
\item For each $i$, we have $A_i \cap \Frac(A_i)^p = A_i^p$.
\end{enumerate}
Then $A$ satisfies condition $(C)$.
\end{proposition}

\begin{proof}
Let $K=\Frac(A)$, $K_i=\Frac(A_i)$, and $K_{\cJ}=\Frac(A_{\cJ})$. Note that $A_i$ and $A_{\cJ}$ are subrings of $A$, and thus domains. Now, let $\epsilon_1, \ldots, \epsilon_s \in K$ be given. Let $\cJ_0 \subset \cI$ be a sufficiently large finite set such that $\epsilon_1, \ldots, \epsilon_s$ belong to $K_{\cJ_0}$. Since $A_{\cJ_0}$ satisfies condition~$(C)$ by assumption~(a), there exists a non-zero element $b \in A_{\cJ_0}$ such that
\begin{displaymath}
A_{\cJ_0} \cap \sum_{j=1}^s \epsilon_j K_{\cJ_0}^p \subset b^{-p} \sum_{j=1}^s \epsilon_j A_{\cJ_0}^p.
\end{displaymath}
We claim that
\begin{displaymath}
A \cap \sum_{j=1}^s \epsilon_j K^p \subset b^{-p} \sum_{j=1}^s \epsilon_j A^p,
\end{displaymath}
which will verify condition $(C)$ for $A$. To see this, it suffices to prove that
\begin{displaymath}
A_{\cJ} \cap \sum_{j=1}^s \epsilon_j K_{\cJ}^p \subset b^{-p} \sum_{j=1}^s \epsilon_j A_{\cJ}^p,
\end{displaymath}
holds for all finite subsets $\cJ$ of $\cI$ containing $\cJ_0$. Let $I(\cJ)$ be the truth value of this containment. We prove $I(\cJ)$ by induction on $\cJ$, the base case being $\cJ=\cJ_0$. Thus suppose that $I(\cJ)$ holds, and let us prove $I(\cJ')$ where $\cJ'=\cJ \cup \{i\}$ for some $i \in \cI \setminus \cJ$. As $A_{\cJ'}=A_{\cJ} \otimes A_i$, the previous lemma yields
\begin{displaymath}
A_{\cJ'} \cap \left( \sum_{j=1}^s \epsilon_j K_{\cJ'}^p \right) = \left( A_{\cJ} \cap \sum_{j=1}^s \epsilon_j K_{\cJ}^p \right) \otimes \left( A_i \cap K_i^p \right).
\end{displaymath}
The first factor on the right is contained in $b^{-p} \sum_{j=1}^s \epsilon_j A^p_{\cJ}$ by $I(\cJ)$, while the second factor on the right is $A_i^p$, by assumption~(b); thus the right side is contained in $b^{-p} \sum_{j=1}^s \epsilon_j A^p_{\cJ'}$, and so $I(\cJ')$ holds.
\end{proof}

\begin{corollary} \label{cor:normal-tensor}
Let $\bk$ be a semi-perfect field, let $\{A_i\}_{i \in \cI}$ be a family of normal finitely generated $\bk$-algebras, and let $A=\bigotimes_{i \in \cI} A_i$, which we assume to be a domain. Then $A$ satisfies condition $(C)$.
\end{corollary}

\begin{proof}
We apply Proposition~\ref{prop:tensor-C}. Since $A_{\cJ}$ is finitely generated over a semi-perfect field, it satisfies $(C)$ by Proposition~\ref{prop:easyC}, and so condition~(a) holds. Since each $A_i$ is normal, condition~(b) holds.
\end{proof}

\begin{corollary} \label{cor:poly-C}
Any polynomial ring over a semi-perfect field satisfies condition $(C)$.
\end{corollary}

\begin{proof}
Suppose $A=\bk[x_i]_{i \in \cI}$ is a polynomial ring. Then $A=\bigotimes_{i \in \cI} A_i$ where $A_i$ is the univariate polynomial ring $\bk[x_i]$. The result now follows from Corollary~\ref{cor:normal-tensor}.
\end{proof}

\begin{corollary} 
Let $A$ be a polynomial ring over a semi-perfect field, let $K=\Frac(A)$ and let $\cU$ be an infinite set. Then $\bR^{\flat}=K \otimes_A A \invl x_i \invr_{i \in \cU}$ is a polynomial $K$-algebra.
\end{corollary}

\begin{proof}
This follows from Corollary~\ref{cor:poly-C} and Theorem~\ref{thm:inverse2}.
\end{proof}

Having given some examples where $(C)$ holds, we now give an example where it does not.

\begin{example}
Let $A=\bF_p[t_i^2,t_i^3]_{i \ge 1}$, a subring of the polynomial ring $\bF_p[t_i]_{i \ge 1}$. We claim that $A$ does not satisfy condition $(C)$. Take $s=1$ and $\epsilon_1=1$, and let $S \subset A \cap K^p$ be the set $S=\{t_i^p\}_{i \ge 1}$. We note that $K=\Frac(A)$ is the rational function field $\bF_p(t_i)_{i \ge 1}$, so $t_i^p$ does indeed belong to $K^p$. We claim that $S$ is not contained in $b^{-p} A^p$ for any non-zero $b \in A$. Indeed, $b$ can only use finitely many variables, say $t_1, \ldots, t_n$, and then $b^{-p} A^p$ cannot contain $t_{n+1}^p$, as $b^{-1} A$ does not contain $t_{n+1}$. We do not know if $\bR^{\flat}$ is a polynomial ring in this case: this is an interesting open problem.
\end{example}

\begin{remark}
Broadly speaking, the main difficulty with the derivational criteria like Theorem~\ref{thm:poly} or those in \cite[\S 2]{stillman} lies in dealing with elements that look like $p$th powers but are not. For example, if $K=\bF_p(t_i)_{i \ge 1}$ then the element $f=\sum_{i \ge 1} t_i x_i^p$ of $K\vars$ looks like a $p$th power, in the sense that it is annihilated by all continuous derivations, even though it is not. The criterion in \cite{stillman} is not powerful enough to handle such elements, and thus cannot establish the polynomiality of $K\vars$. Our Theorem~\ref{thm:poly} can handle this kind of element, and does prove that $K\vars$ is polynomial. However, Theorem~\ref{thm:poly} cannot handle certain more subtle elements. For example, if $A=\bF_p[t_i^2,t_i^3]_{i \ge 1}$ is the ring from the previous example and $g \in \bR^{\flat}$ is the element $g=\sum_{i \ge 1} t_i^p x_i^p$ then $g$ \emph{really} looks like a $p$th power---e.g., it is a limit of $p$th powers, and is a $p$th power in the overring $\bR$---even though $g$ is not a $p$th power in $\bR^{\flat}$.  Our Theorem~\ref{thm:poly} is not powerful enough to handle this kind of element, and thus cannot determine whether or not $\bR^{\flat}$ is a polynomial ring.
\end{remark}

\subsection{A curious example}

We now give an example of a ring $A$ for which $(C_{\kappa})$ holds for $\kappa=\aleph_0$, but fails for $\kappa=\aleph_1$. Thus the ring $\bR^{\flat}$ is polynomial if $\cU$ is countable, but our methods cannot determine whether or not it is so for uncountable $\cU$.

Fix a well-ordered set $I$ of type $\omega_1$. Thus $I$ is uncountable (of cardinality $\aleph_1$), but for any $i \in I$ the set $\{j \in I \mid j<i\}$ is countable. In particular, any countable subset of $I$ is bounded from above (if not, $I$ would be a countable union of countably many subsets). Let $\Gamma$ be the group of all functions $I \to \bZ$ with finite support. We totally order $\Gamma$ using the lexicographic order: that is, $\gamma<\gamma'$ if $\gamma(i)<\gamma'(i)$, where $i$ is maximal such that $\gamma(i) \ne \gamma'(i)$. Given a non-zero element $\gamma \in \Gamma$, the {\bf top index} of $\gamma$ is the maximal $i$ for which $\gamma(i) \ne 0$, and the {\bf top value} of $\gamma$ is $\gamma(i)$, where $i$ is the top index. We let $\Gamma_+$ (resp.\ $\Delta$) be the submonoid of $\Gamma$ consisting of~0 and all non-zero elements $\gamma \in \Gamma$ with top value at least~1 (resp.\ at least~2). We note that $\Gamma_+$ can also be described as the set of $\gamma \in \Gamma$ with $\gamma \ge 0$.

Let $L=\bF_p \lpp \Gamma \rpp$ be the ring of Laurent series with coefficients in $\bF_p$ and exponents in $\Gamma$. By definition, an element of $L$ is a formal series $\sum_{\gamma \in \Gamma} c_{\gamma} t^{\gamma}$, where the $t^{\gamma}$'s are formal symbols, the $c_{\gamma}$'s belong to $\bF_p$, and the set $\{ \gamma \in \Gamma \mid c_{\gamma} \ne 0\}$ is well-ordered (under the order on $\Gamma$). Multiplication in $L$ is performed in the usual manner; the condition on the support of the coefficients ensures it is well-defined. For a non-zero element $f=\sum c_{\gamma} t^{\gamma}$ of $L$, we let $v(f)$ be the minimal $\gamma$ for which $c_{\gamma}$ is non-zero. It is well-known that $L$ is a field and $v$ is a valuation on $L$ with value group $\Gamma$. In fact, this is the standard construction of a valuation field with value group $\Gamma$. The valuation ring $\cO_L$ consists of those elements of $L$ that have the form $\sum_{\gamma \in \Gamma_+} c_{\gamma} t^{\gamma}$. See \cite[\S II.3]{fuchs} for further background on these constructions.

Let $A \subset \cO_L$ be the set of elements of the form $\sum_{\gamma \in \Delta} c_{\gamma} t^{\gamma}$. Since $\Delta$ is a submonoid of $\Gamma$, it follows that $A$ is a subring of $L$. Let $K$ be the fraction field of $A$. We note that the valuation of any non-zero element of $A$ belongs to $\Delta$.

\begin{lemma} \label{lem:top-index}
Given $f \in K$ there exists $i \in I$ such that $t^{\delta} f \in A$ for any $\delta \in \Delta$ with top index $>i$.
\end{lemma}

\begin{proof}
It suffices to prove the result for $f=g^{-1}$ with $g \in A$ non-zero. Let $\gamma=v(g)$; we may as well scale $g$ so that $t^{\gamma}$ has coefficient~1. Let $i$ be the top index of $\gamma$. Write $g=t^{\gamma}+g_1+g_2$, where $g_1$ consists of all terms in $g$ with top index $i$ (other than the leading term) and $g_2$ consists of terms with top index $>i$. Then
\begin{displaymath}
g^{-1} = t^{-\gamma}(1+t^{-\gamma} g_1+t^{-\gamma} g_2)^{-1} = t^{-\gamma} \sum_{n,m \ge 0} \binom{n+m}{n} (t^{-\gamma} g_1)^n (t^{-\gamma} g_2)^m.
\end{displaymath}
The terms in $t^{-\gamma} g_2$ have the same top values as the terms in $g_2$; in particular, they are all at least~2, and so $t^{-\gamma} g_2$ belongs to $A$. The terms in $t^{-\gamma} g_1$ have top index at most $i$, and so the same is true for terms in $t^{-\gamma} (t^{-\gamma} g_1)^n$ for any $n$. Thus if $\delta \in \Delta$ has top index $>i$ then all terms in $t^{\delta} \cdot t^{-\gamma} (t^{-\gamma} g_1)^n$ have the same top index and top value as $\delta$, and so this element belongs to $A$. Thus $t^{\delta} g^{-1}$ belongs to $A$.
\end{proof}

\begin{lemma}
$A$ satisfies $(C_{\kappa})$ with $\kappa=\aleph_0$.
\end{lemma}

\begin{proof}
Let $\epsilon_1, \ldots, \epsilon_s \in K$ be given, and let $S=\{f_n\}_{n \ge 1}$ be a countable subset of $A \cap \sum_{j=1}^s \epsilon_j K^p$. By Lemma~\ref{lem:top-index}, for each $n \ge 1$ we can choose some $i_n \in I$ such that $t^{p \delta} f_n \in \sum_{j=1}^s \epsilon_j A^p$ for any $\delta \in \Delta$ with top index $>i_n$. Since $\{i_n\}_{n \ge 1}$ is a countable subset of $I$, it is bounded from above; let $i^* \in I$ be an element such that $i_n<i^*$ for all $n$. Then if $\delta \in \Delta$ has top index $i^*$ we have $t^{p \delta} f_n \in \sum_{j=1}^s \epsilon_j A^p$ for all $n$, and so $S \subset b^{-p} \sum_{j=1}^s \epsilon_j A^p$ with $b=t^{\delta} \in A$. Thus $(C_{\kappa})$ holds.
\end{proof}

\begin{lemma}
$A$ does not satisfy $(C_{\kappa})$ with $\kappa=\aleph_1$.
\end{lemma}

\begin{proof}
For $i \in I$, let $\delta_i \in \Gamma_+$ be the element that is~1 at $i$ and~0 elsewhere. Then $t^{p\delta_i}$ belongs to $A \cap K^p$ for all $i$: indeed, $n\delta_i$ belongs to $\Delta$ for any $n \ge 2$, and so $t^{p\delta_i}$ belongs to $A$, while $t^{\delta_i}=t^{3\delta_i}/t^{2\delta_i}$ belongs to $K$, and so $t^{p\delta_i}$ belongs to $K^p$. Let $S=\{t^{p\delta_i}\}_{i \in I}$; this is a subset of $A \cap K^p$ of cardinality $\kappa$. To prove the lemma, it suffices to show that $S$ is not contained in $b^{-p} A^p$ for any non-zero $b\in A$. Thus suppose a non-zero $b \in A$ is given. Let $\gamma=v(b)$, let $i$ be the top index of $\gamma$, and let $j>i$. We claim that $t^{p\delta_j}$ does not belong to $b^{-p} A^p$, which will complete the proof. Indeed, suppose that it did, and write $b^p t^{p\delta_j}=a^p$ for some $a \in A$. Taking valuations, we find $\gamma+\delta_j=v(a) \in \Delta$. However, $\gamma+\delta_j$ has top value~1, a contradiction.
\end{proof}

\subsection{Ultraproducts} \label{sec:ultra}

We now treat an ultraproduct version of Theorem~\ref{mainthm}. We refer to \cite[\S 4]{stillman} for the relevant background. Let $\cI$ be a set equipped with a non-principal ultrafilter. We regard elements of the ultrafilter as neighborhoods of some hypothetical point $\ast$. Let $\{\bk_i\}_{i \in \cI}$ be a family of fields, let $R_i = \bk_i[x_1,x_2,\dots]$ with the standard grading, and let $\bS$ be the graded ultraproduct of the $R_i$.

\begin{theorem} \label{thm:ultra}
$\bS$ is a polynomial ring.
\end{theorem}

\begin{proof}
  If the ultraproduct of the $\bk_i$'s has characteristic~0, this follows from \cite[Theorem~4.7]{stillman} (which also covers some cases in positive characteristic). Thus assume that this ultraproduct has positive characteristic~$p$. Passing to a neighborhood of $\ast$, we can thus assume that each $\bk_i$ has characteristic~$p$. Each $R_i$ then comes with a standard $F$-factorization, and an admissible set of Hasse derivations (the Hasse derivatives with respect to the variables). We claim that the ultraproduct of these structures endow $\bS$ with an $F$-factorization and an admissible set $\cD$.

First we start with the $F$-factorization. The properties (F1), (F2), (F3) are clearly preserved under taking ultraproduct, so we need to check (F4). Pick $\epsilon_1,\dots,\epsilon_s \in \bS_0$. Suppose we have $f \in \im(\sigma) \cap \sum_{j=1}^s \epsilon_j \im (\phi)$. In particular, we can write $f = \sigma(x)$ and $f = \sum_{j=1}^s \epsilon_j \phi(y_j)$ for some elements $x,y_1,\dots,y_s \in \bS$. Each of $f,x,y_j$ can be represented by a sequence of elements in $R_i$, say $f = (f_i)$, $x = (x_i)$ and $y=(y_{j,i})$, and the same equalities will hold in some neighborhood of $*$. Similarly, we can write $\epsilon_j = (\epsilon_{j,i})$. In particular, we can write $f_i = \sum_{j=1}^s \epsilon_{j,i} z_{j,i}^p$ where $z_{j,i} \in R_i$. Then we have $f = \sum_{j=1}^s \epsilon_j z_j$ where $z_j$ is the ultraproduct of the $z_{j,i}$, and this shows that $\im(\sigma) \cap \sum_{j=1}^s \epsilon_j \im(\phi) \subseteq \sum_{j=1}^s \epsilon_j \bS^p$. As in Remark~\ref{rmk:F4}, the other inclusion follows from the other axioms (or could be proven in a similar way). The proof that the ultraproduct of admissible sets of derivations is still admissible is similar, so we omit it.

The result now follows from Theorem~\ref{thm:poly}.
\end{proof}

\subsection{Small subalgebras} \label{ss:small}

We now explain the application of Theorem~\ref{thm:ultra} to small subalgebras. We begin with some definitions:

\begin{definition}
Let $\bk$ be a field. We say that {\bf small subalgebras exist} for $\bk$ if for all integers $r$ and $d$ there exist an integer $s=s(r,d,\bk)$ with the following property: given homogeneous polynomials $f_1, \ldots, f_r$ of degrees $\le d$ in the polynomial ring $\bk[x_1, \ldots, x_n]$, with $n \ge s$, there exists a regular sequence $g_1, \ldots, g_s$ in $\bk[x_1, \ldots, x_n]$ consisting of homogeneous elements such that $f_1, \ldots, f_r \in \bk[g_1, \ldots, g_s]$.
\end{definition}

\begin{definition}
Let $\sK$ be a class of fields. We say that {\bf small subalgebras exist uniformly} for $\sK$ if small subalgebras exist for all $\bk \in \sK$ and the quantity $s(r,d,\bk)$ can be taken to be independent of $\bk$ for $\bk \in \sK$.
\end{definition}

Ananyan--Hochster \cite[Theorem A]{ananyan-hochster} proved that small subalgebras exist uniformly for algebraically closed fields. In \cite{stillman}, we proved the following:

\begin{theorem} \label{thm:uniform}
Let $\sK$ be a class of fields. Suppose that for every countable sequence $\{\bk_i\}_{i \in \cI}$ of $\sK$, the ultraproduct ring $\bS$ (as defined in the previous section) is polynomial. Then small subalgebras exist uniformly for $\sK$.
\end{theorem}

\begin{proof}
The argument is exactly the same as in \cite[\S 4.3]{stillman}: simply replace ``perfect field'' with ``field in $\sK$'' everywhere.
\end{proof}

We used Theorem~\ref{thm:uniform} and our polynomiality results to prove that small subalgebras exist uniformly for perfect fields. The results of \cite{stillman} can also prove existence of small subalgebras for semi-perfect fields, but not uniformity for this class. Using the superior polynomiality result of this paper (Theorem~\ref{thm:ultra}), we obtain the following improvement:

\begin{theorem}
Small subalgebras exist uniformly for all fields.
\end{theorem}

\end{document}